\documentclass{article}

\usepackage{amsmath, amsfonts, amssymb, amsthm}
\usepackage[inline]{enumitem}
\usepackage{tikz-cd}
\usepackage[left=3.7cm, right=3.7cm, top=3.7cm, bottom=4.2cm]{geometry}

\usepackage{xcolor}

\author{Terry Gannon, Robin Mader and Arturo Pianzola}
\date{\today}

\newtheorem{thm}{Theorem}
\newtheorem{lem}[thm]{Lemma}
\newtheorem{prop}[thm]{Proposition}
\newtheorem{cor}[thm]{Corollary}

\theoremstyle{definition}

\newtheorem{rem}[thm]{Remark}

\def\id{\ensuremath{\operatorname{id}}}

\def\Aut{\ensuremath{\operatorname{Aut}}}
\def\Spec{\ensuremath{\operatorname{Spec}}}

\def\BAut{\ensuremath{\operatorname{\mathbf{Aut}}}}

\def\Ralg{\ensuremath{R\text{-\textup{\textsf{alg}}}}}

\def\GL{\ensuremath{\operatorname{GL}}}
\def\BGL{\ensuremath{\mathbf{G}\mathbf{L}}}

\def\Mu{\ensuremath{\mathcal{M}}}

\title{Representability of the  automorphism group of finitely generated vertex algebras}

\begin{document}

\maketitle

\begin{abstract} We study the automorphism groups attached to a free algebra with multiple, possibly infinitely many, composition laws. As an application, we prove that the automorphism group of finitely generated vertex algebras over noetherian rings are affine group schemes.
\end{abstract}
\medskip


\bigskip

\section{Introduction}Vertex algebras over commutative base rings other than $\mathbb C$ have been an important part of the theory since the beginning. One of the earliest examples is the monster $\mathbb Z[\frac{1}{2}]$-vertex algebra \cite{BR}. Real, integral, or modular vertex algebras have seen great attention since, see e.g.\ \cite{DongRen, Griess, McRae, Mason, Car, Car2, LiMuJiao}. 
For background material on vertex algebras over $\mathbb C$ we refer the reader to Kac' introductory book \cite{Kac}.

Let $V = \bigoplus_{n \in \mathbb Z} V_n$ be a $\mathbb Z$-graded vertex algebra over a commutative ring $R$, as defined in \cite[Def.\ 2.1]{Griess}, where we assume that the graded pieces $V_n$ are finitely generated as $R$-modules. In this note we show that if $R$ is noetherian, and  $V$ is finitely generated as a vertex algebra and locally free as an $R$-module, then the associated automorphism group functor $\BAut(V)$ is locally isomorphic to a closed subgroup of $\BGL_{N,R}$ for some suitable $N \in \Bbb{N}.$\footnote{\,Throughout ``locally" means locally for the Zariski topology.} In particular, $\BAut(V)$ is affine of finite type. 
 For the case when $R$ is a field, this result was obtained by Dong and Griess \cite{DongGr} by different methods. A simplification of their argument that applies to rings will be given in \cite{Car2}.\footnote{\, The author has informed us that many details remain to be checked.} It is worthwhile to point out that unlike \cite{DongGr} and \cite{Car2}, which use properties that are specific to vertex algebras, our method does not: It is a general result that applies to a large family of algebraic objects of which vertex algebras are an example.

\section{$\Mu$-algebras} Let $X$ be a set. Our proof is based on the construction of a certain free $R$-module $A(X)$ with (possibly infinitely many)  products on it with  no relations between them---as such, the methodology generalizes beyond vertex algebras; it will thus be presented. 

The nature of $A(X)$ is a straightforward generalization of that of the $R$-algebra of the free magma on $X$ \cite[Ch. III \S2.6]{Bbk}. We will outline its construction for the sake of completeness and readability. 

Let $\Mu$  be a nonempty set.  A set $M$ together with a collection of ``multiplications"
\[ \left\{ \begin{aligned} M \times M &\to M, \\ (a,b) &\mapsto a_m b \end{aligned} \right\}_{m \in \Mu} \]
is called an \emph{$\Mu$-magma}. A {\it homomorphism of $\Mu$-magmas} is a set map that preserves all $m$-products. 
First we show the existence of an $\Mu$-magma $M(X)$ containing $X$ as a subset, such that for any $\Mu$-magma $M$ and any set map $f \colon X \to M$, there exists a unique homomorphism $\widetilde f \colon M(X) \to M$ extending $f$. This is the {\it free $\Mu$-magma on $X.$}

We  construct $M(X)$ recursively. For any two sets $A$ and $B$, and any $m \in \Mu$, denote the set $A \times \{ m \} \times B$ by $A_m B$. 
For $k \in \mathbb Z_{> 0}$, define $M^k(X)$ via $M^1(X) = X$ and 
\[ M^k(X) = \coprod_{\substack{ 1 \leq j < k \\ m \in \Mu}} M^j(X)_m M^{k-j}(X). \]
Let $M(X) = \coprod_{k \geq 1} M^k(X)$.
For $m \in \Mu$, $a \in M^i(X)$ and $b \in M^j(X)$, we define $a_m b = (a,m,b) \in M^{i+j}(X)$. 
This makes $M(X)$ into an $\Mu$-magma. Given $f \colon X \to M$ as above and $a \in M(X)$, we  define $\widetilde f(a)$ as follows. Say $a \in M^k(X).$ If $k = 1,$ then $a \in M^1(X) =X$ and we set $\widetilde f(a) = f(a).$ Assume that $k > 1$ and that $\widetilde{f}$ has been defined for elements of $M^j(X)$ for $j <k.$ There exists a unique decomposition $a = (b,m,c) \in M^j(X)_m M^{k-j}(X)$, with $m \in \Mu$ and $1 \leq j < k$. We set $\widetilde f(a) = \widetilde f(b)_m \widetilde f(c)$.

Let now $R$ be a commutative ring. 
The notion of $R$-algebra generalizes in a natural way to the $\Mu$-setting.  An \emph{$\Mu$-algebra over $R$} is an $R$-module $A$ together with a collection of $R$-bilinear functions
\[ \left\{ \begin{aligned} A \times A &\to A, \\ (a,b) &\mapsto a_m b \end{aligned} \right\}_{m \in \Mu} \]
	called the {\it $m$-products}. A {\it homomorphism of $\Mu$-algebras} is an $R$-module map that preserves all $m$-products. Standard concepts for $\Mu$-algebras, such as  left, right, and two-sided ideals, are defined in the obvious way. One then has standard results, for example the $R$-module quotient of an $\Mu$-algebra by a two-sided ideal has a natural $\Mu$-algebra structure over $R$.
Vertex algebras over $R$ are quintessential examples of $\Mu$-algebras over $R$ for $\Mu = \mathbb Z$.

There exists an $\Mu$-algebra $A(X)$ over $R$ containing $X$ as a subset with the following universal property: for any $\Mu$-algebra $A$ and any set map $f \colon X \to A$, there exists a unique homomorphism $\widetilde f \colon A(X) \to A$ extending $f$.

One can construct $A(X)$ as the free $R$-module with basis $M(X)$ and $m$-products defined by $R$-bilinear extension of those on $M(X)$, that is, \[ \Bigg(\sum_{a \in M(X)} \alpha_a a\Bigg)_m \Bigg(\sum_{b \in M(X)} \beta_b b \Bigg) = \sum_{a, b \in M(X)} (\alpha_a \beta_b) a_m b\] for $\alpha_a, \beta_b \in R$ almost all equal to zero.
Now given any set map $f \colon X \to A$, we may view the target $\Mu$-algebra $A$ as an $\Mu$-magma to obtain a unique extension $\widetilde f\colon M(X) \to A$. This in turn extends to $A(X)$ because $A$ is an $R$-module, and the resulting map visibly preserves all $m$-products. We call $A(X)$ the {\it free $\Mu$-algebra over $R$ on $X.$}

The $R$-module $A(X)$ contains the free $R$-module $F(X) = \bigoplus_{x \in X} Rx$. For convenience in what follows we denote $F(X)$ simply by $F.$ Viewing an element $\theta \in \GL_R(F)$ as a set map $X \to A(X)$ induces an $\Mu$-algebra map $\widetilde \theta \colon A(X) \to A(X)$ with $\widetilde \theta |_{F} = \theta$.

\begin{lem}\label{Lemma1}
		The above map $ \theta \mapsto \widetilde \theta$ is an injective homomorphism of groups $ \GL_R(F) \to \Aut (A(X)).$ In fact, this map induces a group isomorphism between $\GL_R(F)$ and  \[ \Aut(A(X),F) = \{ \theta \in \Aut(A(X)) \mid \theta(F) = F\}.\] 	
\end{lem}

\begin{proof}
	Let $\sigma, \theta \in \GL_R(F)$. By definition, $(\widetilde \sigma \circ \widetilde \theta)|_X = \sigma \circ \theta |_X$, but there is only one $\Mu$-algebra map extending $\sigma \circ \theta$, namely $\widetilde{\sigma \circ \theta}$.
	Since $\widetilde{\id} = \id$,  the lemma follows. 
\end{proof}

 
\begin{prop}\label{prop2} Let $A$ be an $\Mu$-algebra over $R$ which contains a copy of $F$ as an $R$-submodule. Let $\Aut(A,F) = \{ \theta \in \Aut(A) \mid \theta(F) = F\}.$ If  $F$ generates $A,$ then:

\begin{enumerate}[label = (\roman*)]
	\item \label{p2i} The map 
\begin{equation}\label{tildemap} \Aut(A,F) \to \Aut(A(X), F), \quad \theta \mapsto \widetilde \theta = \widetilde{\theta|_F}. \end{equation}
is a group monomorphism.

\item \label{p2ii}  There exists a unique surjective $\Mu$-algebra homomorphism $\eta\colon A(X) \to A$ that is the identity on $X.$ 

\item \label{p2iii} Let $I$ denote the kernel of $\eta$. The image of the map in \eqref{tildemap} consists of those $\varphi \in \Aut(A(X),F)$ with $\varphi(I) = I$.
\end{enumerate}
\end{prop}

\begin{proof}
	\ref{p2i} The map is clearly a group homomorphism. Since $F$ generates $A$ we can appeal to Lemma \ref{Lemma1} to conclude that it is injective.

	\ref{p2ii} By assumption $A$ has a copy of $F = \bigoplus_{x \in X} Rx.$ We can thus view the elements of $X$ as elements of $A$. Now (ii) is clear by the universal property of $A(X)$ and the assumption that $F$ generates $A.$

	\ref{p2iii} For $\theta \in \Aut(A,F)$ we have $\eta \circ \widetilde \theta = \theta \circ \eta$, because both sides of the equation are $\Mu$-algebra homomorphisms with the same restriction to $X.$
	Thus $\eta \widetilde \theta (I) = \theta \eta (I) = 0$, or equivalently, $\widetilde \theta(I) \subset I$. The same reasoning applies to $\theta^{-1}$ and so $\widetilde \theta(I) = I$.
	Conversely, let $\varphi \in \Aut(A(X),F)$ with $\varphi(I) = I$. 
	Since $\eta$ is surjective,  
	$\varphi$ induces an automorphism $\theta$ on the quotient $A \simeq A(X)/I$ via $\theta \eta = \eta \varphi$.  Since $\eta(F) = F$, we get $\theta(F) = \eta (\varphi(F)) = F$, so that $\theta \in \Aut(A,F)$.
	From $\eta \varphi = \eta \widetilde \theta$ we deduce $\varphi|_F = \widetilde \theta|_F$, whence $\varphi = \widetilde \theta$.
\end{proof}

\section{Automorphisms of $\Mu$-algebras and representability considerations}

If $S$ is a ring extension of $R,$ that is an object of $\Ralg,$ then the $S$-module $A \otimes_RS$ has a natural $\Mu$-algebra structure over $S,$ said to be obtained from that of $A$ by base change. We thus have an $R$-group functor  $\BAut(A)$ that attaches to $S$ in $\Ralg$ the group of automorphisms of $A\otimes_R S$ which respect its structure of $\Mu$-algebra over $S$.
At the level of arrows, if $f\colon S \to S'$ is an $R$-algebra map, we define $\BAut(A)(f)(g)$ via the diagram
\[ \begin{tikzcd}[column sep = 5em] A \otimes_R S \otimes_S S' \ar[r, "g \otimes \id_{S'}"] \ar[d, "\simeq"] & A\otimes_R S \otimes_S S' \ar[d, "\simeq"] \\ A \otimes_R S' \ar[r, "{\BAut(A)(f)(g)}"] & A \otimes_R S' \end{tikzcd} \]
for all $g \in \BAut(A)(S)$. 

We define a subgroup $\BAut(A,F)$ of $\BAut(A)$ by setting $\BAut(A,F)(S)$ to be the subgroup of  $\BAut(A)(S)$ consisting of elements that  preserve the image of $F \otimes_R S$ in $A \otimes_R S.$ 
\begin{prop} \label{sheaf} The functors $\BAut(A)$ and $\BAut(A,F)$ are group sheaves on the flat site $R_{\rm flat}$ of $R.$ {\rm See  \cite[III \S5 1.1]{DG}.}
\end{prop}
\begin{proof} 
	
	It is clear that both functors respect finite products. Let $S \to S'$ be a faithfully flat extension in $\Ralg$.
	We verify that the upper row of the diagram
	\[
		\begin{tikzcd}
			\BAut(A,F)(S) \ar[r] \ar[d, phantom, "\cap"] & \BAut(A,F)(S')  \ar[r,shift left=.75ex]
			\ar[r,shift right=.75ex,swap] \ar[d, phantom, "\cap"] & \BAut(A,F)(S'\otimes_S S') \ar[d, phantom, "\cap"] \\
			\BGL(A)(S) \ar[r] & \BGL(A)(S')  \ar[r,shift left=.75ex]
			\ar[r,shift right=.75ex,swap] & \BGL(A)(S' \otimes_S S')
		\end{tikzcd}
	\]
	is exact. 
	The lower row is exact by \cite[III \S5 1.10]{DG}, so we need only verify that the unique lift $\theta \in \BGL(A)(S)$ of an element $\theta' \in \BAut(A,F)(S')$ in the difference kernel is itself an element of $\BAut(A,F)(S)$. 
	For convenience, denote $A_S = A \otimes_R S$ and denote the image of $F \otimes_R S$ by $F_S \subset A_S$. 
	Under the identification $A \otimes_R S' \simeq A_S \otimes_S S'$ the map $\theta'$ becomes $\theta \otimes \id_{S'}$ and the image of $F \otimes_R S'$ is identified with $F_S \otimes_S S'$.

	Note that the sequence $F_S \xrightarrow{\theta} A_S \to A_S/ F_S$ is exact since it becomes exact after base change to $S'$ by assumption on $\theta'.$
	Thus $\theta$ preserves $F_S$.
Finally, to show that $\theta$ is an $\Mu$-algebra homomorphism we reason as follows. For $a, b \in A_S$ and $m \in \Mu$, we have 
	\[ \theta(a_m b) \otimes 1 = \theta'(a_m b \otimes 1) =  \theta'( (a \otimes 1)_m (b \otimes 1) ) = \theta'(a \otimes 1)_m \theta'(b \otimes 1) = (\theta(a)_m \theta(b)) \otimes 1, \] 
	because $\theta'$ is a homomorphism. Since  $A_S \to A_S \otimes_S S'$ is an injective map of $S$-modules, it follows that $\theta(a_m b)= \theta(a)_m \theta(b)$.\end{proof}

\begin{lem} \label{lem4}
\begin{enumerate}[label = (\roman*)]
	\item \label{l4i} 
There is natural  $R$-group isomorphism $\BGL_R(F) \simeq \BAut(A(X),F).$
\item \label{l4ii}  Assume that $F$ is such that the canonical map $F \otimes_R S \to A \otimes_R S$ is injective for all $S$ in $\Ralg$ (for example, if $F$ is locally a direct summand of $A$). The construction of Proposition \ref{prop2}\ref{p2i} yields an injective $R$-group morphism $\BAut(A,F) \to \BAut(A(X),F) \simeq \BGL_{R}(F).$
\end{enumerate}
\end{lem}
\begin{proof} \ref{l4i} It is clear that for $S$ in $\Ralg,$ $A(X)\otimes_RS$ is canonically isomorphic to the free $\Mu$-algebra over $S$ on $X.$ Moreover, since the elements of $X$ are part of an $R$-basis of $A(X)$, we can  identify $F \otimes_R S$ with its image in $A(X) \otimes_R S.$ Part \ref{l4i} of the lemma now follows from Lemma \ref{Lemma1}. 

	\ref{l4ii} If $F \otimes_R S$ can be identified with its image in $A \otimes_R S,$ \ref{l4ii} follows by applying Proposition \ref{prop2}\ref{p2i} ``argument-by-argument".
\end{proof}

We are now ready to address the representability of $\BAut(A,F).$ 

\begin{thm}\label{autrep}
	Let the notation be as in Proposition \ref{prop2}. Assume that the following conditions hold:
	\begin{enumerate}[label =(\alph*)]
		\item $X = \{x^1, \dots, x^N\}$ is a finite set. 
	
		\item  $A$ is a locally free $R$-module and the canonical map $F \otimes_R S \to A \otimes_R S$ is injective for all $S$ in $\Ralg.$\footnote{\, As already mentioned, $F$ has the desired property whenever it is locally a direct summand of $A.$ In most applications $A$ is free and the elements of $X$  part of a basis of $A.$}
		
		\item  $R$ is noetherian.
	\end{enumerate}
Then  $\BAut(A,F)$ is an affine $R$-group scheme of finite type that is locally isomorphic to a closed subgroup of $\BGL_N$. \end{thm}

\begin{proof} Consider the  exact sequence of $R$-modules ({\it cf.} Proposition \ref{prop2})
\[ \begin{tikzcd} 0 \ar[r] & I \ar[r, "\subset"] & A(X) \ar[r, "\eta"] & A \ar[r] & 0 \end{tikzcd} \]
Let us first show that $I$ is locally free. If $A$ is free, $I$ is projective. Since $A$ is locally free, $I$ is locally projective, hence projective \cite[Theorem 10.95.6]{SP}. Consider an open cover  of $\Spec(R)$ by connected affines $\Spec(R_i).$  If the projective $R_i$-module $I \otimes_R R_i$ is of finite type, it is locally free. If it is not of finite type, it is free by a result of Bass \cite{Bass}. It follows that $I$ is locally free as claimed.

By Proposition \ref{sheaf}, it will suffice to show that there exists a cover of $\Spec(R)$ by standard affine open subsets $\Spec(R_i)$ such that the restriction of $\BAut(A, F)$ to $\Spec(R_i)$, namely $\BAut(A \otimes_R R_i, F \otimes_R R_i)$, is isomorphic to a closed subgroup of $\BGL_{N,R_i}$ for some $N.$ Note that since $I$ is a locally free $R$-module, it is still locally free after base change to $R_i$. In view of this,  the assumption on  $A$  allows us to reduce to the case when $A$ and $I$ are free $R$-modules.  Assume henceforth this to be the case.   

 Let $A' = \bigoplus_{j \in J_A} Rw_j \subset A(X)$ be a section of $\eta,$ namely $\{\eta(w_j)\}_{j \in J_A}$ is a basis of $A.$ Let $\{w_j\}_{j \in J_I}$ be a basis of $I.$ If $S$ is a ring extension of $R$, then $A(X) \otimes_R S$ is isomorphic to the free $\Mu$-algebra on $X$ over $S$.
 By Lemma \ref{lem4} the map $\BAut(A,F) \to \BAut(A(X),F) \simeq \BGL_{N,R}$  allows us to identify $\BAut(A,F)$ with a subgroup functor of $ \BGL_{N,R}.$ The $S$-points of this subgroup consists of $N\! \times \!N$-matrices whose associated maps $A(X) \otimes_R S \to A(X) \otimes_R S$ preserve the kernel of $\eta_S,$ which is precisely $I \otimes_R S$. 
We show that this is a polynomial condition on the coordinate ring of $\BGL_{N,R}.$

We have $A(X) = I \oplus A' = \bigoplus_{j \in J_I \coprod J_A} R w_j$. 
Let $\theta = (\theta_{ij})_{1 \leq i, j \leq N} \in \GL_N(S)$. Then, there exist polynomials $p_{\ell h} \in R[X_{ij} \mid 1 \leq i,j \leq N]$, for $\ell, h \in J_I \coprod J_A$, such that $\widetilde \theta (w_h) = \sum_{\ell} p_{\ell h} ((\theta_{ij})) w_\ell$. 

Indeed, after a change of $R$-basis it is enough to verify that there exist polynomials $q_{b,a} \in R[X_{ij}]$ for $a, b \in M(X)$ such that $\widetilde \theta(a) = \sum_{b \in M(X)} q_{b,a}((\theta_{ij})) b$, and this we can do inductively on $a$.
If $a \in M^1(X) = X$, then $a = x^j$ and $\widetilde \theta(a) = \sum_{i = 1}^N \theta_{ij} x^i$ defines the polynomials $q_{b,a}.$ Explicitly,  $q_{x^i, x^j} = X_{ij}$ and $q_{b,x^j} = 0$ for $b \notin X.$
Suppose $k \geq 2$ with $q_{c,b}$ already defined for any $b \in M^j(X)$ with $1 \leq j < k$ and $c \in M(X)$. Decompose $a \in M^k(X)$ as $a = (b,m,c) = b_m c$ to find
\begin{align*} \widetilde \theta(a) = \widetilde \theta(b) _m \widetilde \theta (c) &= \Bigg( \sum_{b'} q_{b',b}((\theta_{ij})) b' \Bigg)_m \Bigg( \sum_{c'} q_{c',c}((\theta_{ij})) c' \Bigg) \\ &= \sum_{b', c' \in M(X)} q_{b',b}((\theta_{ij})) q_{c',c}((\theta_{ij})) b'_m c' .\end{align*}
Thus $q_{da} = q_{b'b} q_{c'c}$ for any $d$ whose unique decomposition into smaller words is of the form $d = b' _m c'$, and $q_{da} = 0$ whenever the decomposition is of the form $d = b'_{m'} c'$ with $m' \neq m$ or when $d \in M^1(X) = X$.

	This shows that the polynomials $p_{\ell h}$ exist, and that the conditions \[ p_{\ell h} ((\theta_{ij})) = 0 = p_{\ell h} ((\theta_{ij})^{-1}) \] for $h \in J_I$ and $\ell \in J_A$ cut out the subgroup  $\BAut(A,F)$.	 \end{proof}
	
	\begin{rem} The whole argument rests on knowing that $I$ is locally free. Though the assumption that $R$ be noetherian is sufficient, it is not necessary. Other examples are local rings \cite{K}, or any ring for which projective modules are locally free.
	\end{rem}
	
	\begin{rem} One can formulate the same result with a scheme $(X, \mathcal{O}_X)$ as base. It is clear how to define the concept of an $\mathcal{O}_X$-sheaf $\mathcal{A}$ of $\Mu$-algebras.  Let $\mathcal{F}$ be a coherent $\mathcal{O}_X$-submodule of $\mathcal{A}$ that generates $\mathcal{A}$. Assume that $\mathcal{A}$ is locally free (hence quasi-coherent) and that $\mathcal{F}$ is locally a direct summand of $\mathcal{A}.$ If $X$ is locally noetherian, then the group-sheaf $\BAut(\mathcal{A})$ is a group scheme that is affine of finite type over $X,$ and locally isomorphic to a closed subgroup of $\BGL_{N}$ for some $N.$ The proof is by reduction to the affine case. Cover $X$ by open noetherian affine schemes. If  $U = \Spec(R)$ is one of the elements of such covering, then over $U$ our $\mathcal{A}$ corresponds to a locally free $R$-module $A$ that carries the structure of an $\Mu$-algebra over $R,$ and $\mathcal{F}$ to a  submodule $F$ of $A$ that is of finite type, locally a direct summand of $A,$ and generates $A$ . By Theorem \ref{autrep} the restriction of $\BAut (\mathcal{A})$ to $U$ has all the desired properties.
 \end{rem}
	\section{The graded case and applications to vertex algebras}

Let $A$ be an $\Mu$-algebra over $R.$ We will assume that $A$ is graded by an abelian group $\Lambda.$ Thus $A = \bigoplus_{\lambda \in \Lambda} A_\lambda$ where the $A_\lambda$ are $R$-submodules, and for all $a \in A_\alpha, b \in A_\beta$ and $ m \in \Mu,$ we have $a _m b \in A_\lambda$ for some $\lambda \in \Lambda$ depending on $\alpha, \beta$ and $m.$ The grading thus  defines a function $d \colon \Lambda \times \Mu \times \Lambda \to \Lambda$ so that $\lambda$ above is given by $d(\alpha, m, \beta).$ 
\begin{lem}\label{graded} Assume that $A = \bigoplus_{\lambda \in \Lambda} A_\lambda$ is finitely generated, and that each of the $A_\lambda$ is free of finite rank. Then there exists a finite set $X$ with the following properties:

(i) $X$ leads to an $\Mu$-algebra homomorphism  $\eta \colon A(X) \to A$ as in Proposition \ref{prop2}

(ii) There exists on $A(X)$ a $\Lambda$-grading  where each element of $M(X)$ is homogeneous.

(iii)   $\eta \colon A(X) \to A$ is a $\Lambda$-graded homomorphism of $\Mu$-algebras over $R.$ 
\end{lem}
\begin{proof} Choose $\lambda_1, \dots, \lambda_n$ so that $F = A_{\lambda_1} \oplus \cdots \oplus A_{\lambda_n}$ generates $A.$  Let $X = \{x^1, \dots, x^N\}$ be a basis of $F$ consisting of homogeneous elements. Then $F$ is a direct summand of $A$ and we can define the map $\eta \colon A(X) \to A$ that is the identity of $F.$ 

Next we attach to every element $a$ of  the $\Mu$-magma $M(X)$ a degree ${\rm deg}(a)$ in $\Lambda. $ We do this by induction on the length of $a.$ If $a \in M^1(X)$, then $a = x^i \in X.$  By assumption $x^i \in A_{\lambda_j}$ for 
some $1 \leq j \leq n$ and we set ${\rm deg}(a) = \lambda_j.$  Assume $\ell > 1$ and that we have defined the degree for elements of length less than $\ell.$ If $a \in M^\ell(X),$ then $a = b_mc$ for some unique elements $b \in M^i(X), c \in M^j(X)$ and $m \in \Mu$ where $1 \leq i,j < \ell$ and $i + j = \ell.$ If ${\rm deg}(b) =\beta$ and ${\rm deg}(c) =\gamma,$ we set ${\rm deg}(a) = d(\beta,m,\gamma).$ Since the elements of $M(X)$ are an $R$-basis of $A(X),$ we have our desired $\Lambda$-grading on the module $A(X).$ By construction  if $x \in A(X)_\beta,$ $y \in A(X)_\gamma$ and $m \in \Mu$, then $x_my \in A(X)_\lambda$ where $\lambda = d(\beta, m, \gamma).$ The map $\eta$ is thus a $\Lambda$-graded homomorphism of $\Mu$-algebras over $R.$
\end{proof}

The automorphisms of $A$ by definition  preserve the grading; i.e., they are graded automorphisms. We denote this group by $\Aut_{\rm gr}(A).$  It is clear that for any $S$ in $\Ralg$ the $\Mu$-algebra $A \otimes_R S$ over $S$ has a natural $\Lambda$-grading, and that this construction is functorial on $S.$ This leads to the $R$-group functor $\BAut_{\rm gr}(A)$ that is easily shown to be a subgroup sheaf of $\BAut(A)$ by the argument of Proposition \ref{sheaf} ({\it cf.}\@ \cite{MGP}). 

It is not natural to expect representability in general without some kind of finiteness assumption. The following result shows that in the graded case being finitely generated suffices.\footnote{\,Fortunately, being finitely generated is a property shared by a large class of important vertex algebras, for example $C_2$-cofinite vertex algebras.} We illustrate the general reasoning in the case of vertex algebras, and comment on the general case at the end.  Vertex algebras over rings\footnote{\, The concept goes back to Borcherds. The earliest explicit version that we have found is \cite{Griess}, and this is the definition we follow.}
 fall under our situation if we take both $\Lambda$ and $\Mu$ to be $\Bbb{Z}.$

 \begin{thm}\label{repgr} Let $V = \bigoplus_{n \in \Bbb{Z}} V_n$ be a  finitely generated vertex algebra over a ring $R.$ Assume  that each of  the $R$-modules  $V_n$ is of finite type, and that we are in one of the following two situations:

(a) $R$ is noetherian and $V$ is locally free as an $R$-module.

(b) The $V_n$ are free.  

\noindent Then  $\BAut_{\rm gr}(V)$ is an affine $R$-group of finite type which is  locally isomorphic to a closed subgroup of $\BGL_N$ for some $N.$ 
	\end{thm}
	
\begin{proof} Case (a): Since $V$ is locally free, the $V_\lambda$ are locally projective, hence projective. Choose a finite number of these, say $V_{n_1}, \dots , V_{n_\ell},$ that generate $V,$ and let $F = \bigoplus V_{n_i}$ be their sum. This is a projective $R$-module of finite type, hence locally free. By further localizing $R$ we may assume that each of the $V_{n_i}$, hence also $F,$ are free. Note that $F$ is a direct summand of $V$. Let $N$ be its rank. 
 
By Theorem \ref{autrep} we may assume that  $\BAut(V,F)$ is  a closed subgroup of $\BGL_{N,R}$. 
Since any degree preserving automorphism automatically preserves $F$, we find that $\BAut_{\rm gr}(V)$ is a subgroup sheaf of $\BAut(V,F)$. We claim that $\BAut_{\rm gr}(V)$ is in fact a closed subscheme of  $\BAut(V,F).$
Since $F$ generates $V,$ an element of $\BAut(V,F)$ is in $\BAut_{\rm gr}(V)$ if and only if it preserves each of the $V_{n_i}.$ But this is clearly a polynomial condition on the coordinate ring of $\BGL_{N,R}$, hence on that of $\BAut(V,F).$ 
 Since we want the vacuum vector $\mathbf 1$ of $V$ to be preserved, we may have to add an additional polynomial condition.  The same conclusion holds if one restricts to automorphisms fixing any number of specific vectors (for example, automorphisms of vertex operator algebras are often required to fix the conformal vector $\omega$). The same conclusion also holds for  many variants of vertex algebras, such as vertex superalgebras.

	Case (b): Let us write $V = A$ to be consistent with previous notation. By Lemma \ref{graded} we have a $\Bbb{Z}$-grading on $A(X)$ such that the map $\eta \colon A(X) \to A$ is graded. It follows that the kernel $I$ of $\eta$ is a graded ideal $I = \bigoplus_{n \in \Bbb{Z}}I_n$ where $I_n = I \cap A(X)_n.$ Because by definition $(V_i) _{m} (V_j) \subset V_{i + j -m - 1},$ we see that $A(X)_n$ is a free $R$-module of infinite rank for all $n$ (it contains elements of $M(X)$ of arbitrarily large length).  

Consider the  exact sequence of $R$-modules ({\it cf.} Proposition \ref{prop2})
\[ \begin{tikzcd} 0 \ar[r] & I \ar[r, "\subset"] & A(X) \ar[r, "\eta"] & A \ar[r] & 0 \end{tikzcd} \]
as in the proof of Theorem \ref{autrep}. Since $\eta$ is graded, we may chose our section $\sigma$ of $\eta$ so that $A'_n = \sigma(A_n) \subset A(X)_n.$ This yields the split exact sequence of $R$-modules 
\[ \begin{tikzcd} 0 \ar[r] & I_n \ar[r, "\subset"] & I_n \oplus A'_n = A(X)_n \ar[r, "\eta"] & A_n \ar[r] & 0 \end{tikzcd} \]

Since $A(X)_n$ is free and $A'_n$ is free of finite rank, $I_n$ is stably free. But then $I_n$ is free by a theorem of Gabel \cite[Ch. I Prop. 4.2]{Lam} (recall $I_n$ is not finitely generated). It follows that $I$ is free. The proof of Theorem \ref{autrep} shows that $\BAut(A,F)$ is a closed subgroup of $\BGL_{N,R}.$  As explained in Case (a), $\BAut_{\rm gr}(A)$ is a closed subgroup of $\BAut(A,F).$  \end{proof}

The following result was proved by Dong and Griess using  delicate considerations involving the commutator and iterate formulas  of vertex algebras \cite{DongGr}.
\begin{cor} Let $V$ be a finitely generated vertex algebra over a field $k.$ Then $\BAut_{\rm gr }(V)$ is an affine algebraic group over $k.$ \qed
\end{cor}

\begin{rem} Theorem \ref{repgr} can be considered for an arbitrary finitely generated $\Lambda$-graded $\Mu$-algebra $A = \bigoplus_{\lambda \in \Lambda} A_\lambda$ over a ring $R$ for which the graded components $A_\lambda$ are of finite type. Case (a) holds unchanged. If the $A_\lambda$ are free, the graded considerations apply and one is once again reduced to the need that $I$ is (locally) free. A useful situation where this is the case is when the $I_\lambda$ are not finitely generated, so one can appeal to  Gabel's theorem.  Whether this holds depends of course on the grading and on $A,$ but it is safe to say that for any $\Mu$-algebra with enough ``relations,"  like vertex algebras, this will be the case.
\end{rem}

\bibliographystyle{amsplain}
\bibliography{refs}

\small{	\noindent Mathematisches Institut, Ludwig-Maximilians-Universität München, Theresienstr.\@ 39, 80333 München, Germany; and \\ 
Department of Mathematical and Statistical Sciences, University of Alberta, Edmonton, Alberta T6G 2G1, Canada, \\
\textit{e-mail}: \textsf{rmader@ualberta.ca}

\medskip
\noindent Department of Mathematical and Statistical Sciences, University of Alberta, Edmonton, Alberta T6G 2G1, Canada, \\
\textit{e-mail}: \textsf{tjgannon@ualberta.ca}

\medskip
\noindent Department of Mathematical and Statistical Sciences, University of Alberta, Edmonton, Alberta T6G 2G1, Canada, \\
\textit{e-mail}: \textsf{a.pianzola@ualberta.ca}
}

\end{document}